\documentclass[11pt]{article}
\usepackage{multirow}
\usepackage{authblk}
\usepackage{amsmath,amssymb,amsbsy,amsfonts,amsthm,latexsym,
            amsopn,amstext,amsxtra,euscript,amscd,color}
            
\PassOptionsToPackage{hyphens}{url}\usepackage{hyperref}
    
\usepackage[capbesideposition=outside,capbesidesep=quad]{floatrow}

 \usepackage{fullpage, array}

\newtheorem{theorem}{Theorem}[section]
\newtheorem{corollary}[theorem]{Corollary}
\newtheorem{lemma}[theorem]{Lemma}
\newtheorem{conjecture}[theorem]{Conjecture}
\newtheorem{proposition}[theorem]{Proposition}
\theoremstyle{definition}

\newtheorem{question}[theorem]{Question}

\DeclareMathOperator{\des}{des}

\def\G{{\mathcal G}}
\def\+{\oplus}

\renewcommand{\S}{\mathcal S}

\def\00{{\bf 0}}
\def\11{{\bf 1}}
\def\+{\oplus}

\def\\{\cr}
\def\({\left(}
\def\){\right)}

\usepackage{booktabs}

\usepackage[colorinlistoftodos,prependcaption,textsize=tiny]{todonotes}

\providecommand{\newoperator}[3]{%
  \newcommand*{#1}{\mathop{#2}#3}}

\newoperator{\FD}{\mathrm{FD}}{\nolimits}

\title{Descents in powers of permutations}
\author[1]{Kassie Archer}
\author[1]{Aaron Geary}
  
\affil[1]{\small{Department of Mathematics, United States Naval Academy, Annapolis, MD, 21402, 

Email: karcher@usna.edu, geary@usna.edu}}

\date{}
\begin{document}

\maketitle
\begin{abstract}
We consider a few special cases of the more general question: How many permutations $\pi\in\mathcal{S}_n$ have the property that $\pi^2$ has $j$ descents for some $j$? In this paper, we first enumerate Grassmannian permutations $\pi$ by the number of descents in $\pi^2$. We then consider all permutations whose square has exactly one descent, fully enumerating when the descent is ``small"  and providing a lower bound in the general  case. Finally, we enumerate permutations whose square or cube has the maximum number of descents, and finish the paper with a few future directions for study.

\end{abstract}

\section{Introduction and Background} 


Let $\S_n$ denote the symmetric group on $[n]=\{1,\ldots,n\}$. Each element of this group is a bijection that can be written in its one-line notation as $\pi = \pi_1\pi_2\ldots\pi_n$ where $\pi_i=\pi(i)$ for all $i\in[n]$. A \emph{descent} in a permutation $\pi$ is an index $i$ with $1\leq i\leq n-1$ such that $\pi_i>\pi_{i+1}$, and we denote by $\des(\pi)$ the number of descents the permutation $\pi$ has. It is well known that the set of permutations with a given number of descents is enumerated by the Eulerian numbers, $A(n,k)$, which have the property that the corresponding Eulerian polynomial $A_n(x) = \sum_{k=0}^{n-1} A(n,k)x^{k+1}$ satisfies
\[
\dfrac{A_n(x)}{(1-x)^{n+1}}=\sum_{m\geq0} (m+1)^nx^m.
\]
There have been many papers answering enumerative questions regarding descents, see \cite{P15}. 

A \textit{Grassmannian permutation} is a permutation that has at most one descent, and is alternatively characterized as those permutations avoiding the patterns 321, 2143, and 3142. We denote by $\G_n$ the set of Grassmannian permtuations in $\S_n$. These permutations are related to Schubert varieties (see for example, \cite{M91}), but have also been studied as combinatorial objects in their own right, as in \cite{GT23,GT22,MS23}.

In this paper, we answer a few special cases of the following general question regarding a permutation $\pi$ and its square in the symmetric group, denoted $\pi^2$.
\begin{question}
How many permutations $\pi\in\S_n$ have the property that $\pi^2$ has $j$ descents with $j \in \{0,1,\ldots,n-1\}$?
    Additionally, how many permutations $\pi\in\S_n$ have the property that $\pi$ has $k$ descents and $\pi^2$ has $j$ descents with $k,j \in \{0,1,\ldots,n-1\}$?
\end{question}
Some related results have already been studied. In particular, many papers have been dedicated to the enumeration of permutations $\pi$ with $k$ descents so that $\pi^2$ has zero descents; these are exactly involutions with $k$ descents (see \cite[A161126]{OEIS}). Results regarding these permutations can be found in \cite{C08,GR93,S81}, among others. 

In  Section~\ref{grass}, we enumerate the set of Grassmannian permutations whose square has $k$ descents for each possible value of $k.$ We note that these results can also be understood in terms of strong avoidance or chain avoidance as in \cite{AG24,BS19}. A permutation \emph{strongly avoids} a pattern if that permutation and its square both avoid a pattern. Chain avoidance generalizes this idea; a permutation \emph{avoids the chain} $(\sigma:\tau)$ if $\pi$ avoids $\sigma$ and $\pi^2$ avoids $\tau.$ As an example, the set of Grassmannaian permutations whose square is also Grassmannian is exactly the set of permutations that strongly avoid the patterns 321, 2143, and 3142. In Theorem~\ref{thm:grassmannian}, we find that the number of such permutations is $\binom{n+1}{3}+1.$
In Section~\ref{cubes}, we address a similar question for Grassmannian permutations whose cube is Grassmannian, presenting a conjecture for the actual answer. 

In Section~\ref{onedes}, we consider permutations whose square has only one descent in a particular position, or of a certain size, and enumerate these permutations for certain special cases. In particular, we enumerate permutations that have only one descent in the first position or that have at most one ``small" descent, and we find bounds on the total number of permutations whose square is Grassmannian. In Section~\ref{decreasing}, we enumerate permutations whose square or cube is the decreasing permutation (i.e., the square or cube has $n-1$ descents). Finally, in Section~\ref{future}, we present some interesting directions for future study and some conjectures. 

\section{Squares of Grassmannian Permutations}\label{grass}

In this section, we enumerate the set of Grassmanian permutations whose square has at most $k$ descents for each possible value of $k$. The results of this section are summarized in Table~\ref{tab:grassm}. We first present the following lemma, which shows that $0\leq k\leq 3,$ and we then proceed by cases.

\begin{lemma}\label{lem:grass-3}
    If $\pi\in \G_n$, then $\des(\pi^2)\leq 3$. 
\end{lemma}

\begin{proof}
Clearly the identity permutation satisfies this property, so it is enough to consider a permutation $\pi$ with exactly one descent at position $i\in[n-1].$ Then $\pi$ is composed of 2 ascending runs, $\pi_1\ldots\pi_i$ and $\pi_{i+1}\ldots\pi_n$. Thus $\pi^2$ starts with $\pi_{\pi_1}\cdots\pi_{\pi_i}$ which can have at most one descent because $\pi$ has only one descent. Also $\pi^2$ ends in $\pi_{\pi_{i+1}}\cdots\pi_{\pi_n}$ which  has at most one descent. Finally, there may also be a descent between $\pi_{\pi_{i}}$ and $\pi_{\pi_{i+1}}$, and thus the total amount of descents is 3 or fewer. 
\end{proof}

We will utilize the next lemma to enumerate Grassmanian permutations whose square have at most $k$ descents for any $k$, and then later in the next section will use it again to talk about Grassmannian cubes with a certain number of descents.
\begin{lemma}\label{lem:recur}
For $n\geq 3$ suppose that $a_n$ is the number of permutations $\pi\in\G_n$ with the property that $\pi^r$ has at most $k$ descents, and $b_n$ is the number of such permutations with the additional property that $\pi_1\neq 1$ and $\pi_n\neq n$. Then we must have that $a_n = 2a_{n-1}-a_{n-2}+b_n$.
\end{lemma}

\begin{proof}
   First, we note that the number of such permutations that have 1 as a fixed point is $a_{n-1}$ since 1 cannot be part of any descent in $\pi$ or $\pi^r$. Similarly, the number of such permutations that have $n$ as a fixed point is $a_{n-1}$. Since there are $a_{n-2}$ such permutations that have both as fixed points, we know there are $2a_{n-1}-a_{n-2}$ permutation $\pi\in\G_n$ with the property that $\pi^r$ has at most $k$ descents and that have either 1 or $n$ as a fixed point of $\pi$. Thus the total number must satisfy the recurrence $a_n = 2a_{n-1}-a_{n-2}+b_n$. \end{proof}











\renewcommand{\arraystretch}{2}
 \begin{table}[ht]
            \begin{center}
\begin{tabular}{| >{\centering\arraybackslash} m{1cm} | >{\centering\arraybackslash} m{5cm} | >{\centering\arraybackslash} m{3cm} | >{\centering\arraybackslash} m{3.5cm} |}
    \hline
    $k$ & $|\{\pi\in\G_n:\des(\pi)\leq k\}|$ & Theorem & OEIS \\
    \hline
    \hline
   0 &
   $\left\lfloor \dfrac{n^2}{4} \right\rfloor+1$ &
    Theorem~\ref{Grass-inv}  & A033638\\[4pt]
    \hline
    1 & $\displaystyle\binom{n+1}{3} + 1$ & Theorem~\ref{thm:grassmannian} & A050407\\[4pt]
    \hline
    2 & $\displaystyle2F_{n+3}-\left\lfloor\frac{(n+4)^2}{4}\right\rfloor+1$ & Theorem~\ref{Grass-2}  & A001588 $-$ A002620\\[4pt]
    \hline
    3 & $2^n-n$ &  Corollary~\ref{cor:lessthan3} &A000325 \\[4pt]
    \hline
\end{tabular}
            \end{center}
            \caption{Number of Grassmannian permutations whose square has at most $k$ descents}
            \label{tab:grassm} 
      \end{table}

We first consider the number of permutations $\pi\in\G_n$ so that $\des(\pi^2)=0$, i.e., when the square of $\pi$ is the identity permutation. These are exactly involutions with a single descent together with the identity permutation. One can derive this result from \cite[Theorem 9.2]{GR93}, but we present an independent proof here that  resembles the other proofs in this section.
\begin{theorem}\label{Grass-inv}
    For $n\geq 1$, the number of permutations in $\G_n$ that are involutions is equal to $\lfloor\frac{n^2}{4}\rfloor+1$.
\end{theorem}

\begin{proof}
    Let $a_n$ denote the number of permutations $\pi\in\G_n$ with $\des(\pi^2)=0$. We will show that $a_n$ satisfies the recurrence $a_n=2a_{n-1}-a_{n-2}+1$ when $n$ is even and $a_n=2a_{n-1}-a_{n-2}$ when $n$ is odd. Together with the initial conditions $a_1=1$ and $a_2=2$ has closed form $\lfloor\frac{n^2}{4}\rfloor+1$. By Lemma~\ref{lem:recur}, it is enough to show that there is exactly one involution $\pi\in\G_n$ so that $\pi_1\neq 1$ and $\pi_n\neq n$ when $n$ is even, and zero when $n$ is odd.

    Note that since $\pi$ only has one descent, it must be of the form $\pi_1\pi_2\ldots \pi_{k-1}n1\pi_{k+2}\ldots \pi_n$, where $\pi_i>i$ for $i\leq k$ and $\pi_i<i$ for $i\geq k+1$, and thus there are no fixed points, so $n$ must be even. Furthermore, each element before the descent will be in a 2-cycle with each element after the descent, so the descent must occur at position $k=n/2.$ Since $\pi_n=k$ and $\pi_1=k+1$, there is only one such permutation that works.
\end{proof}

\begin{theorem}\label{thm:grassmannian}
    For $n\geq 2$, the number of permutations $\pi\in\G_n$ with $\pi^2\in\G_n$ is equal to $\binom{n+1}{3}+1.$
\end{theorem}

\begin{proof}
    If $a_n$ is the number of such permutations, then it is enough for us to show that for $n\geq 3$, we have the recurrence $a_n = 2a_{n-1}-a_{n-2}+n-1$, since $\binom{n+1}{3}+1$ satisfies this recurrence and the initial conditions that $a_1=1$ and $a_2=2$.
By Lemma~\ref{lem:recur}, it is enough to show that there are $n-1$ permutations $\pi\in\G_n$ with $\des(\pi)\leq 1$ that do not start with $\pi_1=1$ and do not end with $\pi_n=n.$

 We must have that  all Grassmannian permutations that do not start with $\pi_1=1$ and do not end with $\pi_n=n$ look like \[\pi = \pi_1\ldots \pi_{i-2}n1\pi_{i+1}\ldots \pi_n\] where $\pi_{i-1}=n$, $\pi_i=1$, $\pi_1<\pi_2<\ldots<\pi_{i-1}$ and $\pi_i<\pi_{i+1}<\ldots<\pi_n.$    We claim that all such permutations whose square is also Grassmannian must look like $\pi=k(k+1)\ldots n12\ldots (k-1)$ for some $k\in\{2,3,\ldots, n\}.$ To prove this, let us assume not for the sake of contradiction. 

    If $\pi$ is not of the form described above, we must have that there is some $j$ with $i\leq j<n$ so that $\pi_j<\pi_1<\pi_{j+1}\leq \pi_n.$ Therefore in $\pi^2,$ we have \[\pi^2 = \pi_{\pi_1} \ldots \pi_{\pi_{i-2}} \pi_n\pi_1\ldots\pi_{\pi_n}.\]
    Since $\pi_n>\pi_1,$ there is a descent in position $i-1$ of $\pi^2.$ Since $\pi^2_i=\pi_1\neq 1$, this descent does not involve the element 1, and so we must have $\pi^2_1=1$ in order to have no other descents in $\pi^2.$ Thus, $\pi_1=i$. But then we must have that $\pi_j=i-1$, and thus $\pi^2_j=n$, so there must be another descent at position $j<n$, which is a contradiction. 

    Therefore, there are $n-1$ Grassmannian permutations with $\pi_1\neq 1$ and $\pi_n\neq n$ whose square is also Grassmannian, and the result follows.
\end{proof}

\begin{theorem}\label{Grass-2}
    For $n\geq 2$, the number of permutations $\pi\in\G_n$ with $\des(\pi^2)\leq 2$ is equal to $2F_{n+3}-\Big\lfloor\frac{(n+4)^2}{4}\Big\rfloor+1$. 
\end{theorem}
\begin{proof}
    If $a_n$ is the number of such permutations, the it is enough for us to show that $a_n = 2a_{n-1}-a_{n-2}+
    2F_{n-1}$ when $n$ is odd and $a_n = 2a_{n-1}-a_{n-2}+
    2F_{n-1}-1$ when $n$ is even, since the solution to this recurrence is exactly the formula listed in the theorem.
    By Lemma~\ref{lem:recur}, we only need to show that the number of such permutations with $\pi_1\neq 1$ and $\pi_n\neq n$ is equal to $2F_{n-1}$ when $n$ is odd and $2F_{n-1}-1$ when $n$ is even.

    First we claim that if $\pi_i=n$ with $i<n/2$, then we must have $\pi_1>i$. For the sake of contradiction, suppose that $\pi_1\leq i$. We will show that $\pi^2$ has three descents. 
    Notice that since $\pi_i=n$, we must have $\pi_{i+1}=1$. Since $\pi_1\leq i<n/2$, we must have that $\pi_n>\pi_1$, and so there is a descent in $\pi^2$ at position $i$. We will now see that there is a descent in a position greater than $i$ and a descent in a position less than $i$. 
    Since $i<n/2$, we must have that $\pi_n>i.$ Therefore there must be some $j>i$ so that $\pi_j<i<\pi_{j+1}$. However, we must also have that $\pi_{j+1}<j+1$ since $\pi_n<n$. Thus $\pi_{\pi_{j+1}}\leq \pi_j$ and since $\pi_1>1$, we also have $\pi_{\pi_j}>\pi_j$. This implies that $\pi^2_{j}>\pi^2_{j+1}$, which is a descent in a position greater than $i$. 
    Now, since $\pi_1\leq i<n/2$, there must be some $k<i$ so that $r=\pi_k\leq i$ and $t=\pi_r> i$. 
    Since $\pi_n<n$, we must have $\pi_t<t=\pi_r$. Therefore, $\pi^2_k>\pi^2_r$, so there must be a descent before position $i.$ Since we have found three descents, the claim is proven.

    Now let us see that if $\pi_i=n$ with $i<n/2$, then $\pi$ must be of the form
    \[\pi = \pi_1\ldots\pi_{i-1}n1\pi_{i+2}\ldots \pi_n\] where $\{\pi_1,\ldots,\pi_{i-1}\}$ are any subset of the elements of $\{i+1,\ldots,n-1\}$. Indeed, if we take $\pi^2$, there cannot be a descent before position $i$, and there can only be one descent after position $i$. Including the possible descent at position $i$, this means there is a maximum of $2$ possible descents. Therefore there are $\binom{n-i-1}{i-1}$ such permutations. 

    A symmetric argument to the one above proves that there are  $\binom{n-i-1}{i-1}$ permutations so that $\pi_i=n$ with $i>n/2.$ Therefore, we only need to consider the last case when $i=n/2.$ 
    
    We claim that there is only one such permutation in this case, namely the permutation $\pi = (n/2+1) \ldots n1\ldots (n/2).$ Since $\pi^2$ is the identity in this case, this permutation clearly satisfies the requirement. We need only show that if $\pi_1<\pi_n$, such a permutation will contain 3 descents. For the sake of contradiction, suppose $\pi_1<\pi_n.$ This implies that $\pi_1\leq \frac{n}{2}=i$ and $\pi_n>\frac{n}{2}=i$, so a proof similar to the one above shows there must be a descent at a position before and after $i$. Since there will also be a descent in $\pi^2$ at position $i$, $\pi^2$ will indeed have 3 descents. 

    Thus we have shown that there are \[2\sum_{i=1}^{\lfloor \frac{n-1}{2}\rfloor} \binom{n-1-i}{i-1}\] Grassmannian permutations whose square has at most 2 descents and with $\pi_1\neq 1$ and $\pi_n\neq n$, plus one extra permutation when $n$ is even. Therefore, there are exactly $2F_{n-1}$ permutations when $n$ is odd and $2(F_{n-1}-1)+1$ permutations when $n$ is even and the statement of the theorem follows. 
\end{proof}

We remark that the number Grassmannian permutations $\pi$ such that $\pi^2$ has at most 3 descents is equal to the total number of Grassmannian permutations by Lemma~\ref{lem:grass-3}. Since it is well-known and easily checked that the number of such permutations is $2^n-n,$ we have the following corollary of Lemma~\ref{lem:grass-3}.
\begin{corollary}\label{cor:lessthan3}
     The number of Grassmannian permutations of length $n$ whose square has less than or equal to 3 descents is equal to $2^n-n$. 
\end{corollary}

By Theorem \ref{Grass-2}, the number of Grassmannian permutations such that $\pi^2$ has fewer than 3 descents grows like the Fibonacci numbers, which have a growth rate of $\frac{1+\sqrt{5}}{2}\approx 1.618$.  Since the total number of Grassmannian permutations is $2^n-n$, we see that as $n$ grows almost all Grassmannian permutations have the property that its square has exactly three descents. As an example, when $n=20$ there are 1,048,556 Grassmannian permutations and 991,385 of these have squares with exactly 3 descents, a ratio of approximately 0.94.

\section{Cubes of Grassmannian Permutations}\label{cubes}

In this section we consider cubes of Grassmannian permutations. We note that similar to Lemma~\ref{lem:grass-3}, which shows the square of a Grassmannian permutation can have at most 3 descents, for a similar reason a cube of a Grassmannian permutation can have at most 7 descents. 

Let us first consider Grassmannian permutations whose cube has zero descents, i.e.,~those Grassmannian permutations of order 3. 
 We note that one can derive this result using \cite[Theorem 9.3]{GR93}, which uses quasi-symmetric generating functions derived from a bijection to orderings on binary necklaces, but we also present an independent proof here.

\begin{theorem}
    For $n\geq 2$, number of permutations $\pi\in\G_n$ with $\des(\pi^3)=0$  is 
\begin{equation*}
    \binom{\lfloor\frac{n}{3}\rfloor+3 }{3} + \binom{\lfloor\frac{n-1}{3}\rfloor+3 }{3} + \binom{\lfloor\frac{n-2}{3}\rfloor+3 }{3} - n
    \end{equation*}
    which is equivalent to
\[
 \begin{cases}
 \frac{1}{2}(k^3+4k^2-k+2) & n=3k, \\
 \frac{1}{2}(k^3+5k^2+2k+2) & n=3k+1, \\ 
 \frac{1}{2}(k^3+6k^2+5k+2) & n=3k+2.
\end{cases}
\]
\end{theorem}





\begin{proof}
    Let $a_n$ be the number of permutations of $n$ elements with one descent and with cubes having 0 descents. We will show that $a_n$ satisfies the recurrence
\begin{equation*}
    a_n=2a_{n-1}-a_{n-2} + \begin{cases}
        \frac{n}{3}+1 & n = 0 \mod 3 \\
        0 & n = 1,2 \mod 3
    \end{cases}
\end{equation*}
 which with initial conditions $a_1=a_2=1$ has a closed form given in the statement of the theorem. By Lemma~\ref{lem:recur}, it is enough for us to show that if $\pi_1\neq 1$ and $\pi_n\neq n$, there are exactly $\frac{n}{3}+1$ such permutations when $n$ is a multiple of 3 and zero otherwise.

First, consider a permutation composed only of 3-cycles which has at most one descent. Since Grassmanian permutations form a permutation class, it must be the case that any subsequence of this permutation has at most one descent as well. Let's suppose that $(a,b,c)$ and $(d,e,f)$ are two cycles in the permutation decomposition of $\pi$ where $a<b<c$, $d<e<f$, and $a<d$. Then we must have that $a<d<b<e<c<f$ since otherwise the subsequence of $\pi$ involving these 6 elements would have more than one descent. For example, if $a<b<d<e<c<f$, then they would appear in $\pi$ as $bcefad$ which has a descent in both positions 2 and 4. 

This implies that for any 3-cycles of the form $(a_i,b_i,c_i)$ with $a_i<b_i<c_i$, we must have that $a_j<b_k<c_\ell$ for all $j,k,\ell$. In the one-line notation of the permutation, all elements $b_i$ will appear before each $c_i,$ which in turn will appear before each $a_i$.

Similarly, if we have cycles of the form $(x_i,z_i,y_i)$ with $x_i<y_i<z_i$, we must have $x_j<y_k<z_\ell$ for all $j,k,\ell$. In the one-line notation of the permutation, all elements $z_i$ will appear before each $x_i,$ which in turn will appear before each $y_i$.

Finally, for any two cycles $(a,b,c)$ and $(x,z,y)$, we must have $a<b<x<c<y<z$ since otherwise the sequence obtained by these six elements will contain more than 1 descent. Thus, our permutation must have the form 
\[\pi=b_1b_2\ldots b_rc_1c_2\ldots c_rz_1z_2\ldots z_ma_1a_2\ldots a_rx_1x_2\ldots x_m y_1y_2\ldots y_m\]
where $\pi$ is composed of the cycles $(a_i,b_i,c_i)$ and cycles $(x_i,z_i,y_i).$ Since $r$ can be any number between $0$ and $n/3$, there are $n/3+1$ possible permutations of this form.

Finally, any fixed points of $\pi$ must appear at the beginning or end of the permutation, so in fact, the permutations described above are all the permutations that do not have $\pi_1=1$ or $\pi_n=n.$
\end{proof}

Unfortunately the proof above does not generalize to permutations whose cube is Grassmannian since such a permutation does not necessarily have a given cycle type. However, we present a conjecture for the number of Grassmannian permutations whose cube is also Grassmannian in terms of a rational generating function.
\begin{conjecture}
    If $b_n$ is the number permutations $\pi\in\G_n$ with $\pi^3\in\G_n$, and $B(x) = \sum_{n\geq 0} b_n x^n$, then \[
    B(x) = \frac{1-x+x^2+5x^4+2x^5+3x^6-x^7}{(1-x^3)^2(1-x)^2}.
    \] 
\end{conjecture}
This conjecture has been checked up to $n=20.$
To prove this conjecture, by Lemma~\ref{lem:recur}, it would be enough to show that the number of Grassmanian permutations $\pi$ with $\pi_1\neq 1$ and $\pi_n\neq n$  so that $\pi^3$ has exactly one descent is equal to
\[
\begin{cases}
    n-3 & n=0 \bmod 3\\
    n-1 & n=1,2 \bmod 3.
\end{cases}
\]

\section{Permutations whose squares have one descent}\label{onedes}

In this section we consider all elements of $\S_n$ whose squares have exactly one descent. We start by characterizing permutations with the property that their squares have one descent occurring in a particular position (namely, the first) and use this to find bounds for any position. Then we characterize permutations which have one descent of a given size, fully enumerating permutations whose square has one ``small" descent and providing bounds for any descent size. 

Let's first define $e_n$ to be the number of involutions in $\S_n$, described by OEIS \cite{OEIS}  A000085. Note that these numbers can be expressed with the formula 
\[
e_n=\sum_{k=0}^{\lfloor \frac{n}{2}\rfloor} \binom{n}{2k}(2k-1)!!
\]
and has the exponential generating function $E(x) = \text{exp}(x+x^2/2).$ We will also need the notion of direct sums in this section. The \emph{direct sum} of two permutations $\sigma\in\S_k$ and $\tau\in\S_\ell$ is the permutation $\pi= \sigma\oplus \tau\in\S_{k+\ell}$ defined by  \[\pi_i = \begin{cases}
    \sigma_i & \text{for } 1\leq i\leq k\\
    \tau_{i-k}+k & \text{for } k+1\leq i\leq k+\ell.
\end{cases}\]
For example, $45123\oplus 231 = 45123786$. Note that by \cite[Lemma 2.1]{AG24}, the powers of direct sums are the direct sums of powers. That is, if $\pi = \bigoplus_{i=1}^k \sigma_i$, then $\pi^m= \bigoplus_{i=1}^k \sigma_i^m$. 

\subsection{Position of the descent}
We first consider the number of permutations whose square has one descent in a given position $i.$ In particular, we exactly enumerate those permutations whose square has exactly one descent in position 1 and find a lower bound for the number of permutations that have a descent in each other position.
We will need the following lemma.
\begin{lemma}\label{lem:1descent}
    Let $\pi \in \S_n$. If $\pi^2=n12\cdots(n-1)$, then $n$ is odd and
    \begin{equation*}
        \pi=\Big(\frac{n+1}{2}\Big)\Big(\frac{n+1}{2}+1\Big)\cdots(n-2)(n-1)n12\cdots\Big(\frac{n+1}{2}-2\Big)\Big(\frac{n+1}{2}-1\Big).
    \end{equation*}
\end{lemma}

\begin{proof}
      Assume $n$ is in position $k$ of $\pi$, i.e., $\pi_k=n$. If $k=1$, then it must be that $\pi_n=n$ in order for $\pi^2_1$ to be $n$. The only permutation that satisfies both of these conditions is the trivial case of $\pi=1$.  Assume $k=2$. Then to have $\pi^2_1=n$ it must be that $\pi_1=2$, which implies $\pi$ starts with $2n\pi_3$ and $\pi^2$ starts with $n(\pi_n)\pi_{\pi_3}$. In order for this to be the correct form it must be that $\pi_n=1$ and $\pi_{\pi_3}=2$ which implies $\pi_3=1$. Therefore $\pi=231$ is the only permutation whose square has the desired form when $k=2$. 
      
    Continuing for $k\geq3$, if $\pi_k=n$, then $\pi_1=k$. And since $\pi_n$ is in the $k$-th position of $\pi^2$, it must be that $\pi_n=k-1$ and $\pi_{k-1}=n-1$ in order for $\pi^2_n$ to be $n-1$. And since $\pi^2_{k-1}$ must be $k-2$ it must be that $\pi_{n-1}=k-2$. This argument continues until we see $\pi_{k+1}=1$ which implies $\pi_2=k+1$. Therefore $\pi$ must start with an increasing, consecutive run from $k$ to $n$ followed by an increasing, consecutive run from $1$ to $k-1$. And since $n$ is in position $k$, the length of the first increasing sequence is $k$ and the second is $k-1$. Thus $k+k-1=n$ and $k=\frac{n+1}{2}$. This is only possible when $n$ is odd and there is only one permutation whose square has the desired form for each $k$.  
\end{proof}




\begin{theorem}\label{thm:1descent}
       The number of permutations $\pi\in\S_n$ whose squares have exactly one descent in position 1 is equal to 
    \begin{equation*}
        \sum_{i=1}^{\lfloor \frac{n-1}{2}\rfloor}  e_{n-(2i+1)}.
    \end{equation*}
\end{theorem}

\begin{proof}
    Let $\alpha_m$ be the permutation in Lemma \ref{lem:1descent} of length $m$, let $u_n$ be an involution of length $n$, and let $\iota_n$ be the identity permutation of length $n$. We claim any permutation whose square has exactly one descent occurring in the first position must be of the form 
\begin{equation*}
    \pi = \alpha_m \oplus u_{n-m}.
\end{equation*}

First, notice 
\begin{align*}
    \pi^2 &= \alpha_m^2 \oplus u_{n-m}^2 \\
    &= \alpha_m^2 \oplus \iota_{n-m} \\ 
    &= m12\cdots (m-1) \oplus 12\cdots (n-m) \\
    &= m12\cdots (m-1)(m+1)\cdots n,
\end{align*}
which has one descent occurring the first position and no others.

Next, assume there is some $\tau$ such that $\tau^2 = m12\cdots (m-1)(m+1)\cdots n$. We claim that $\tau_{m+1}$ through $\tau_n$ must consist of elements all greater than $m$. If not, then there is some $k>m$ such that $\tau_j=k$ with $j<m$. Then it must be that $\tau_k=j-1$ in order for $\tau^2_j$ to be $j-1$. But then we also have $\tau_{j-1}=k$ so that $\tau^2_k=k$, which is a contradiction since we also have $\tau_j=k$. Thus there cannot be any element less than $m$ after position $m$ in $\tau$. And if that is the case, then it must be that $\tau=\sigma\oplus \eta$ with $\sigma \in S_m$ and $\eta \in S_{n-m}$. Since $\tau^2 = \sigma^2 \oplus \eta^2$, by Lemma \ref{lem:1descent} it must be that $\sigma^2=\alpha_m$. And if $\tau^2$ must end in a consecutive increasing run after $\tau^2_m$, it must that $\eta$ is an involution. Therefore, in order for $\tau^2$ to be a permutation with one descent in the first position, it must be of the form $\alpha_m \oplus u_{n-m}$.

Finally, since there is only one $\alpha_m$ for each odd value of $m$, the number of permutations whose square has one descent in position one must be the same as the involutions of size $n-m$ for all odd $m\geq 3$ up to $m=n$ for odd $n$ and $m=n-1$ for even $n$. Thus the total number of elements in $\S_n$ whose squares have one descent in the first position is $\displaystyle \sum_{i=1}^{\lfloor \frac{n-1}{2}\rfloor}  e_{n-(2i+1)}$.
\end{proof}

As an example of Theorem \ref{thm:1descent} consider the case of $n=11$, which is broken down by the size of $\alpha_m$ in Table \ref{tab:n11}. Recall $e_n$ (in the fourth column) is the number of involutions in $\S_n$. Summing the fourth column, the total number of permutations of 11 elements whose squares have one descent in position 1 is 853.

\renewcommand{\arraystretch}{1}
\begin{table}[ht]
            \begin{center}
\begin{tabular}{| >{\centering\arraybackslash} m{1cm} | >{\centering\arraybackslash} m{3cm} | >
{\centering\arraybackslash} m{3cm} | >{\centering\arraybackslash} m{3cm} | }
    \hline
    $i$ & $m=2i+1$ & $\alpha_m$ & $e_{11-m}$ \cite{OEIS} \\
    \hline
    \hline
    1 & 3 & $231$ & 764 \\
    \hline
    2 & 5 & 34512 & 76 \\
    \hline 
    3 & 7 & 4567123 &  10 \\
    \hline
    4& 9 & 567891234 & 2 \\
    \hline
    5& 11 & 6789(10)(11)12345 & 1 \\
    \hline
\end{tabular}
            \end{center}
            \caption{An example of the breakdown of permutations of length 11 whose squares have exactly one descent in the first position, as described in the proof of Theorem~\ref{thm:1descent}. We find that in total, there are 853 such permutations by adding the numbers in the last column.}
            \label{tab:n11} 
      \end{table}

We also note that a symmetric argument to the one above (by considering the reverse-complement of the permutations) implies the following. 
\begin{corollary}\label{cor:n minus 1}
           The number of permutations $\pi\in\S_n$ whose squares have exactly one descent in position $n-1$ is equal to 
    \begin{equation*}
        \sum_{i=1}^{\lfloor \frac{n-1}{2}\rfloor}  e_{n-(2i+1)}.
    \end{equation*}
\end{corollary}

Finally, we note the following lower bound on permutations whose square has exactly one descent in another position. 

\begin{proposition}\label{prop:des-lower}
    The number of permutations $\pi\in\S_n$ whose squares have exactly one descent in position $k$ for any $2\leq k\leq n-2$ is bounded below by
    \begin{equation*}
        \sum_{i=1}^{\lfloor \frac{n-1}{2}\rfloor}  e_{n-(2i+1)}.
    \end{equation*}
\end{proposition}

\begin{proof}
    Note that if $\gamma\in\S_{n-(2i+1)}$ is an involution and $\alpha_{2m+1}$ is defined as above, then consider the permutation $\pi = \gamma'_1\gamma'_2\ldots \gamma'_{k-1} \alpha'_{2m+1}\gamma'_k\ldots \gamma'_{n}$, where the elements of $\alpha'_{2m+1}$ is the permutation of the set $\{k,k+1\ldots,k+2m\}$ in the same relative order as $\alpha_{2m+1}$ and $\gamma'$ is the permutation of the remaining elements of $[n]$ in the same relative order as $\gamma.$ Notice this permutation has the property that $\pi^2=\iota_{k-1}\oplus\alpha^2_{2m-1}\oplus\iota_{n-k-2m+1}$ and so has exactly one descent in position $k.$
\end{proof}
 We note that not all such permutations $\pi\in\S_n$ whose square has exactly one descent in position $k$ are of this form. For example, the permutation $\pi =36571824$ has the property that $\pi^2= 58123467$ has one descent in position 2, but is not explained by the proof of the proposition above.

As a corollary of Theorem~\ref{thm:1descent}, Corollary~\ref{cor:n minus 1}, and Proposition~\ref{prop:des-lower}, we have the following lower bound on the total number of permutations whose square has one descent.
\begin{corollary}
    A lower bound for the number of permutations $\pi\in\S_n$ whose squares have exactly one descent is \[(n-1)\sum_{i=1}^{\lfloor \frac{n-1}{2}\rfloor}  e_{n-(2i+1)}.\]
\end{corollary}

In Proposition \ref{prop:bound} of the following subsection, we arrive at a better lower bound using a different approach by considering all possible sizes of the single decent. 

\subsection{Size of the descent}

Here, we consider all permutations whose squares have exactly one descent of a given size $k$; that is, if the descent is in position $i$, then $\pi_i-\pi_{i+1}=k$. First we will show the size of the descent cannot be 1. 

\begin{lemma}
    Let $\pi\in \S_n$. If $\pi^2$ has one descent in position $i$, then $\pi^2_i \geq \pi^2_{i+1}+2$. 
\end{lemma}

\begin{proof}
    Assume $\pi^2$ has one descent of size 1 in position $i$. Then it must be of the form $12\cdots (i+1)i\cdots(n-1)n$. This is an odd permutation with the single transposition $(i, i+1)$. However, since squares in $\S_n$ are necessarily even, then it must be that $\pi^2$ is not a square, which is a contradiction. Therefore if there is one descent in $\pi^2$ it must be of size at least 2.
\end{proof}

We need to introduce more notation before proceeding to the general case. 
Let $\mathcal{A}_k\subset \S_{k+1}$ denote the set of permutations that have exactly 1 descent of size $k$. That is, for some $\alpha\in\mathcal{A}_{k}$, we have that $(k+1)1$ appears consecutively and there are no other descents. Let $\mathcal{B}_k\subset\S_{k+1}$ denote the ``square roots'' of permutations in $\mathcal{A}_k$. That is, for any $\beta\in\mathcal{B}_k$, we have $\beta^2\in\mathcal{A}_k$. Finally, let $b_k=|\mathcal{B}_k|$. For example, for $k=2$, we have $\mathcal{A}_2=\{231,312\}$ and $\mathcal{B}_2 = \{231,312\}$, so $b_2=2$. Meanwhile for $k=3$, we have $\mathcal{A}_3=\{2341,2413,3412,4123\}$ and $\mathcal{B}_3=\{2341,4123\}$ since $2341^2=4123^2=3412$, so $b_3=2$ as well. Notice that in this case, only one element of $\mathcal{A}_3$ has a square root.


\begin{theorem}\label{thm:sizek}
    For  $n\geq 3$, the number of permutations $\pi\in\S_n$ whose squares have one descent of size $k$ with $2 \leq k \leq n-1$ is
    \begin{equation*}
        (n-k)b_{k} e_{n-k-1}.
    \end{equation*}
\end{theorem}

\begin{proof}
Let $\pi^2\in\S_n$ be a permutation with one descent of size $k$ in position $i$ and let $\pi^2_i=m$. Then $\pi^2_{i+1}=m-k$. Furthermore, we must have that $\pi^2=12\ldots (m-k-1)\pi_{m-k}\ldots\pi_{m}(m+1)(m+2)\ldots n$ where $\pi_{m-k}\ldots\pi_m$ is a permutation on the set $\{m-k,\ldots, m\}$ in the same relative order to some element of $\mathcal{A}_k$. In other words, $\pi^2$ must be of the form $\iota_{m-k-1} \oplus \alpha \oplus \iota_{n-m}$ where $\alpha\in\mathcal{A}_k$ and $\iota_j$ is the identity permutation of length $j$. 

 This implies our answer must be the number of ``square roots" of permutations of the form $\iota_j \oplus \alpha \oplus \iota_\ell$ for some $\alpha\in\mathcal{A}_k$. 
 We first claim that $\pi$ must also have the property that $\{\pi_{m-k}\ldots\pi_m\}=\{m-k,\ldots, m\}$. Indeed, if there were some $m-k\leq j\leq m$ with $\pi_j=\ell$ with $\ell<m-k$ or $\ell>m$, then we would have to have that $\pi_\ell=j$ in order for $\pi^2_\ell=\ell$ as required. However, this would imply that $\pi^2_j=j$ which must be a contradiction since $\pi^2_{m-k}\ldots\pi_{m}$ must be in the same relative order as an element of $\mathcal{A}_k$, all of which must not have any fixed points. Permutations in $\mathcal{A}_k$ will always have the property that elements that appear before the descent are excedances and elements that appear after the descent are nonexcedances. 
 
Therefore $\pi_{m-k}\ldots\pi_{m}$ must be in the same relative order as an element $\beta\in\mathcal{B}_k$. Removing these elements and considering the remaining elements, it is clear that $\pi_1\ldots\pi_{m-k-1}\pi_{m+1}\ldots \pi_n$ must be in the same relative order as an involution since its square must be the identity permutation.
Therefore the permutation $\pi$ is determined by an element $\beta\in\mathcal{B}_k\subset\S_{k+1}$, an involution of length $n-k-1$, and $m$. Since there are $n-k$ choices for $m$, the result follows. 
\end{proof}

Since we can compute $b_k$ for small values of $k$, we can exactly enumerate all permutations whose square has a small descent. These are presented in Table~\ref{tab:smallk}.

\begin{table}[ht]
            \begin{center}
\begin{tabular}{| >{\centering\arraybackslash} m{1.2cm} | >{\centering\arraybackslash} m{1cm} | >
{\centering\arraybackslash} m{9.5cm} | }
    \hline
   Descent size $k$ & $b_{k}$ & $|\{\pi\in\S_n: \des(\pi^2)=1 \text{ where the descent is of size $k$}\}|$ \\
    \hline
    \hline
    2 & 2 & $2(n-2)e_{n-3}$  \\
    \hline
    3 & 2 & $2(n-3)e_{n-4}$ \\
    \hline 
    4 & 6 & $6(n-4)e_{n-5}$  \\
    \hline
    5& 12 &  $12(n-5)e_{n-6}$\\
    \hline
    6& 22 & $22(n-6)e_{n-7}$ \\
    \hline
    7& 48 & $48(n-7)e_{n-8}$ \\
    \hline
    8& 108 & $108(n-8)e_{n-9}$  \\
    \hline
    9& 186 & $186(n-9)e_{n-10}$  \\
    \hline
    10& 366 & $366(n-10)e_{n-11}$  \\
    \hline
    11& 912 & $912(n-11)e_{n-12}$  \\
    \hline
    12& 1506 & $1506(n-12)e_{n-13}$  \\
    \hline
\end{tabular}
            \end{center}
            \caption{The number of permutations whose square has exactly one descent of size $k$ for any $2\leq k\leq 12$. Here $e_j$ denotes the number of involutions of size $j$, as given in OEIS A000085 \cite{OEIS}.}
            \label{tab:smallk} 
      \end{table}

The next result follows immediately from the fact that the size of the descent can range from 2 to $n-1$. 

\begin{corollary}\label{cor:onedescent}
    The number of permutations $\pi\in\S_n$ whose squares have one descent is 
    \begin{equation*}
       \sum_{k=2}^{n-1}  b_{k+1} e_{(n-(k+1))}(n-k).
    \end{equation*}
\end{corollary}

As an example of Corollary \ref{cor:onedescent}, consider $\pi \in \S_{9}$ such that $\pi^2$ has 1 descent, broken down by descent size in Table \ref{tab:n9}. Summing the fourth column, we see there are 2204 permutations of length 9 whose squares have one descent.

\begin{table}[ht]
            \begin{center}
\begin{tabular}{| >{\centering\arraybackslash} m{1.2cm} | >{\centering\arraybackslash} m{1cm} | >
{\centering\arraybackslash} m{2.5cm} | >{\centering\arraybackslash} m{3.75cm} | }
    \hline
   Descent size $k$ & $b_{k}$ & $e_{9-k-1}$ & Product of column~2, column~3, and $(9-k)$ \\
    \hline
    \hline
    2 & 2 & 76 & 1064 \\
    \hline
    3 & 2 & 26 & 312 \\
    \hline 
    4 & 6 & 10 & 300  \\
    \hline
    5& 12 & 4 & 192 \\
    \hline
    6& 22 & 2 & 132 \\
    \hline
    7& 48 & 1 & 96 \\
    \hline
    8& 108 & 1 & 108 \\
    \hline
\end{tabular}
            \end{center}
            \caption{Breakdown of permutations of length 9 whose squares have exactly one descent by the size of the descent. There are 2204 such permutations, obtained by summing the last column.}
            \label{tab:n9} 
      \end{table}

Though we don't have a formula describing $b_k$ exactly, we can easily bound it below by $2\lfloor\frac{k}{2}\rfloor$ by considering those Grassmanian permutations whose square has one descent of size $k$, as described in the proof of Theorem~\ref{thm:grassmannian}. These permutations are exactly of the form $\pi = r(r+1)\ldots k12\ldots(r-1),$ except in the case when $r=k/2$, in which case the square has zero descents. This gives us the following lower bound on all permutations whose square has exactly one descent.

\begin{proposition}\label{prop:bound}
     A lower bound for the number of permutations $\pi\in\S_n$ whose squares have exactly one descent is \[\sum_{k=2}^{n-1}2\bigg\lfloor\frac{k}{2}\bigg\rfloor(n-k)  e_{n-(k+1)}.\]
\end{proposition}

\section{Monotone Decreasing Squares and Cubes}\label{decreasing}

In this section, we consider permutations whose powers have as many descents as possible; that is, those permutations $\pi$ so that $\pi^k$ is the monotone decreasing permutation for some $k$. Let us first consider $k=2$. Note that this is equivalent to enumerating permutations avoiding the chain $(\varnothing:12)$ in the notation of \cite{AG24}. 

\begin{theorem}
    The number of permutations whose square is the monotone decreasing permutation is equal to 
    \[
     \begin{cases}
        \dfrac{(2m)!}{m!} & \text{ if } n=4m \text{ or } n=4m+1 \\ 
        0 & \text{ otherwise.}
    \end{cases}
    \]
\end{theorem}
\begin{proof}
    First, let us note that since any square of a permutation must be even, and the decreasing permutation is only even if $n$ is 0 or 1 mod 4, there are no permutations whose square is the monotone decreasing permutations when $n$ is 2 or 3 mod 4.

    Let us consider the case where $n=4m$ for some positive integer $m$. Then for any $j\leq n/2,$ we must have that there is some $a$ so that $\pi(j)=a$ and $\pi(a) = n+1-j$ and there is some $b$ so that $\pi(n+1-j) = b$ and $\pi(b) = j$. Therefore, we must have that $(j,a,n+1-j,b)$ is a 4-cycle in the permutation $\pi$ for each choice of $j$. Furthermore we must have that either $a=n+1-b$ or $b=n+1-a.$ 

    The question we must answer is: how many ways can we insert pairs of pairs of the form $\{j,n+1-j\}$ into four-cycles. It is clear that the number of ways partition the elements from $[\frac{n}{2}]$ into $m$ subsets of size 2 is equal to \[\frac{(2m)!}{2^m m!}.\] Since for each cycle associated to $\{i,j\}\in[\frac{n}{2}]$, we have two possible choices for 4-cycles, namely $(i, j, n+1-i, n+1-j)$ or $(i,n+1-j,n+1-i,j)$, we multiply this number by $2^m$ and the result follows.

    Finally, if $n=4m+1,$ the same argument works for any element except $\frac{n+1}{2}$, which must be a fixed point. 
\end{proof}

We can also enumerate those permutations whose cube is the decreasing permutation.  This is equivalent to enumerating permutations avoiding the chain $(\varnothing:\varnothing:12)$ in the notation of \cite{AG24}. 

\begin{theorem}
        The number of permutations whose cube is the monotone decreasing permutation is equal to 
    \[
     \sum_{i=0}^{\lfloor \frac{n}{6}\rfloor} \binom{\lfloor \frac{n}{2}\rfloor}{3i}\frac{(3i)!}{i!3^i}\cdot 4^i
    \]
\end{theorem}
\begin{proof}
    First, let us note that any permutation whose cube is the decreasing permutation must be composed of only 6-cycles, 2-cycles, and a single fixed point when $n$ is odd. Indeed, $\pi^3$ is the decreasing permutation, $\pi^6$ is the identity. Furthermore, we cannot have any 3-cycles, since this would imply that $\pi^3$ had at least 3 fixed points, which it does not. The only fixed point is the element $\frac{n+1}{2}$, when $n$ is odd. 

    Clearly any 2-cycle must be composed of an element $j$ and its counterpart $n+1-j.$ Additionally, any 6-cycle must be composed of three elements $j,a,b$ and their counterparts $n+1-j,n+1-a,$ and $n+1-b.$ 

    There are exactly 8 types of 6-cycles that will return the decreasing permutation when cubed, namely: 
    \begin{center}
        \begin{tabular}{cccc}
        (1, 2, 3, 6, 5, 4) &(1, 2, 4, 6, 5, 3) &(1, 3, 5, 6, 4, 2) &(1, 3, 2, 6, 4, 5)\\
(1, 4, 5, 6, 3, 2) &(1, 4, 2, 6, 3, 5) &(1, 5, 4, 6, 2, 3) &(1, 5, 3, 6, 2, 4)
        \end{tabular}
    \end{center}
    Finally, if we want to enumerate the permutations whose cube is the decreasing permutation which is composed of $i$ 6-cycles, we need to select three elements $\{i,a,b\}$ for each 6-cycle from the first $\lfloor\frac{n}{2}\rfloor$, of which there are $\binom{\lfloor \frac{n}{2}\rfloor}{3i}$ possibilities. Then we sort them into the $i$ different 3-cycles; there are $\frac{(3i)!}{i!(3!)^i}$ ways to do this. Finally, we choose of the the 8 types for each 6-cycle. Since there is only one way for the remaining 2-cycles to be configured (each element $j$ with its counterpart $n+1-j$), there are exactly \[\binom{\lfloor \frac{n}{2}\rfloor}{3i}\frac{(3i)!}{i!(3!)^i}\cdot 8^i\] permutations with $i$ 6-cycles. Taking the sum from 0 to $\lfloor \frac{n}{6}\rfloor,$ the result follows. 
\end{proof}

We finish this section with the next natural question to ask regarding these permutations.

\begin{question}
    How many permutations are there whose $k$-th power is the decreasing permutation for $k>3$?
\end{question}

\section{Concluding Remarks and Open Questions}\label{future}


There are many potential directions of study following this work. We would like a formula for $b_k$ in Theorem \ref{thm:sizek}, which would provide an exact answer to the question of how many permutations have squares that are Grassmannian. Short of that goal, it would be nice to improve upon the bounds of $b_k$. Additionally, one could continue to extend upon these results and enumerate permutations whose squares have 2, 3, and up to $n-2$ descents. One could also fix a different number of descents in $\pi$, say 2 or 3, and investigate the distribution of descents in $\pi^2$ or higher powers.

\vspace{1cm}
\begin{table}[ht]
            \begin{center}
\begin{tabular}{c|ccccccccc|c}
$n \backslash k$  & 0 & 1 & 2 & 3 &4 &5 &6 & 7 & 8 & $\mathbb{E}(\des(\pi^2))$ \\ \hline
2  & 2& 0& &&&&&&&0  \\
3  & 4& 2& 0  & &&&&&& 1/3 \\
4  & 10& 6& 6& 2  &&&&&& 1 \\
5  &  26& 22& 48& 22& 2 &&&&&  8/5 \\
6  &  76& 68& 276& 260& 40& 0 &&&&  13/6 \\
7  &  232& 214& 1384& 2204& 944& 62& 0 &&& 19/7  \\
8  &  764& 672& 6240& 16172& 13212& 3048& 200& 12  &&  13/4\\
9  &  2620& 2204& 27096& 103588& 145160& 70740& 10936& 524& 12 &  34/9 \\
10 & 9496& 7354& 113722& 612178& 1338370& 1145614& 364366& 36838& 862 &  43/10 
\end{tabular}
 \caption{Distribution of permutations of length $n$ by the number of descents in their squares, together with the weighted average over all permutations. Here, $\mathbb{E}(\des(\pi^2)) = \frac{1}{n!}\sum_{\pi\in\S_n} \des(\pi^2)$.}
            \label{tab:exp} 
\end{center}
\end{table}
One might also consider the expected number of descents in $\pi^2$. It is obvious that the expected number of descents for a permutation $\pi\in\S_n$ is $\frac{n-1}{2}$, but it appears that the answer for the expected number of descents $\pi^2$ has over all permutations $\pi\in\S_n$ also has a nice answer. Table~\ref{tab:exp} includes the number of permutations $\pi\in\S_n$ with  $\des(\pi^2)=k$ for $2\leq n \leq 10$, along with the expected number of descents, demonstrating the conjecture for these values of $n$.

\begin{conjecture}
   For $n\geq 3,$ the expected number of descents in $\pi^2$ is \[\frac{n-1}{2}- \frac{2}{n}.\]
\end{conjecture}

Perhaps surprisingly, we have the same conjecture for expected number of descents in $\pi^3$. 


\end{document}